\documentclass[12pt,oneside]{amsart}

\usepackage{hyperref,graphicx}
\usepackage[scale=0.7]{geometry}
\usepackage{color}

\newcommand{\abs}[1]{\lvert#1\rvert}

\newcommand{\F}{\mathcal{F}}
\newcommand{\K}{\mathcal{K}}
\newcommand{\G}{\mathcal{G}}
\newcommand{\Z}{\mathbb{Z}}
\newcommand{\R}{\mathbb{R}}
\newcommand{\N}{\mathbb{N}}
\newcommand{\Q}{\mathbb{Q}}
\newcommand{\Pp}{\mathcal{P}}
\newcommand{\real}{\mathbb{R}}

\newcommand{\cH}{\mathcal{H}}

\newcommand{\bigsetcond}[2]{\bigl\{ #1 \,:\, #2 \bigr\}}

\DeclareMathOperator{\inte}{int}

\DeclareMathOperator{\conv}{conv}

\newtheorem{theorem}{Theorem}[section]

\newtheorem{prop}[theorem]{Proposition}
\newtheorem{lemma}[theorem]{Lemma}
\newtheorem{coro}[theorem]{Corollary}
\newtheorem*{conj}{Conjecture}
\theoremstyle{definition}
\newtheorem{definition}[theorem]{Definition}
\newtheorem{example}[theorem]{Example}
\theoremstyle{remark}

\numberwithin{equation}{section}

\begin{document}
\title[$S$-Helly numbers]{Helly numbers of Algebraic Subsets of $\R^d$}

\author{J. A. De Loera}
\author{R. N. La Haye}
\author{D. Oliveros}
\author{E. Rold\'an-Pensado}
\address[J. A. De Loera, R. N. La Haye]{Department of Mathematics, UC Davis}
\email{deloera@math.ucdavis.edu, rlahaye@math.ucdavis.edu}
\address[D. Oliveros, E. Rold\'an-Pensado]{Instituto de Matem\'aticas, UNAM campus Juriquilla}
\email{dolivero@matem.unam.mx, e.roldan@im.unam.mx}


\begin{abstract} We study $S$-convex sets, which are the geometric objects obtained as the intersection of the usual convex sets in $\R^d$ with a proper subset $S\subset \R^d$.  We contribute new results  about their $S$-Helly numbers. We extend  prior work for $S=\R^d$, $\Z^d$, and $\Z^{d-k}\times\R^k$; we give sharp bounds on the $S$-Helly numbers in several new cases. We  considered the situation for low-dimensional $S$ and for sets $S$ that have some algebraic structure, in particular when $S$ is an arbitrary subgroup of $\R^d$ or when $S$ is the difference between a lattice and some of its sublattices.  By abstracting the ingredients of  Lov\'asz method we obtain colorful versions of many monochromatic Helly-type results, including several colorful versions of our own results.  
\end{abstract}

\maketitle

\section{Introduction}
Eduard Helly stated his fundamental theorem slightly under a century ago \cite{originalHelly}.
It says that the members of a finite family of convex sets in $\R^d$ intersect, if every $d+1$ of them intersect (see \cite{Mbook} for the basic introduction to this part of convexity). Since its discovery Helly's theorem has found many generalizations, extensions and applications in many areas of mathematics
(see  \cite{DGKsurvey,de2015quantitative,Eckhoffsurvey93,Wen1997} and references therein). Continuing the work of many authors (see e.g., \cite{AW2012,Doignon:1981fv,hoffman,Jamison:1981wz, Kay:1971uf, convexityspaces-vandevel} and the many references therein), our paper presents new versions of Helly's theorem where the intersections in the hypothesis and conclusions of the theorems are restricted to exist in a proper subset $S$ of $\R^d$.
The key challenge is to compute the associated $S$-Helly number.  It is fair to say that applications in optimization have prompted already many papers about $S$-Helly numbers \cite{AW2012,Bel1976,clarkson,conforti-unp,ViolatorSpaces2008,hoffman,Sca1977}. The interest in our new $S$-Helly numbers is motivated not just by old and even new (see e.g., \cite{chanceopt2015}) applications to optimization, but also in relation to algebraic structures for $S$ and new applications to algebraic computation with polynomials \cite{algrandomsampling}, where more sophisticated values for the variables appear.

The 1970s and 1980s saw large growth in the research of \emph{abstract convexity} where
Helly-type theorems were explored in abstract settings beyond Euclidean spaces.  In this article, given a proper subset $S \subset \R^d$, 
we consider the \emph{convexity space}  whose convex sets are the intersections of the standard convex sets in $\R^d$ with $S$ (see \cite{arochabrachomontejano,Doignon:1981fv,hoffman,Jamison:1981wz, Kay:1971uf, convexityspaces-vandevel} and the many references therein for more on abstract convexity spaces). For example, there is a Helly-type theorem that talks about the existence of intersections over the \emph{integer lattice} $\Z^d$, proved by Doignon \cite{Doi1973} (later rediscovered in \cite{Bel1976,hoffman,Sca1977}); it states that a finite family of convex sets in $\R^d$ intersect at a point of $\Z^d$ if every $2^d$ of members of the family intersect at a point of $\Z^d$. A second example is the work of Averkov and Weismantel \cite{AW2012} who gave a \emph{mixed} version of Helly's and Doignon's theorems which includes them both. This time the intersection of the convex sets is required to be in $\Z^{d-k}\times\R^k$ and this can be guaranteed if every $2^{d-k}(k+1)$ sets intersect in such a point (their formula had been previously stated by A.J. Hoffmann in \cite{hoffman}).  Note that the size of subfamilies  that guarantees an intersection of all members of the family in $S$ depends on $S$. These are the \emph{Helly numbers of $S$} from our title, which we now define precisely.

For a nonempty family $\mathcal K$ of sets, the \emph{Helly number} $h=h(\mathcal K)\in\N$ of $\mathcal K$ is defined as the smallest number satisfying the following:
\begin{align}\label{helly:cond}
	& \forall i_1,\ldots,i_h \in [m] : F_{i_1} \cap \cdots \cap F_{i_h} \neq \emptyset && \Longrightarrow && F_1 \cap \cdots \cap F_m \neq \emptyset
\end{align}
for all $m \in\N$ and $F_1, \ldots, F_m \in \mathcal K$. If no such $h$ exists, then $h(\mathcal K):=\infty$. E.g., for the traditional Helly's theorem, $\mathcal K$ is the family of all convex subsets of $\R^d$.

Finally for $S \subseteq \real^d$ we define
\begin{equation*} 
	h(S) := h\bigl( \bigsetcond{S \cap K}{K \subset \real^d \ \text{is convex\,} }\bigr).
\end{equation*}
That is, $h(S)$ is the Helly number when the sets are required to intersect at points in $S$; we will call this the \emph{$S$-Helly number}. We stress again that a set $S$ and its convex sets $ \K_S=\bigsetcond{S \cap K}{K \subset \real^d \ \text{is convex}}$ form a convexity space. The key problem we study here is: given a set $S$, bound $h(S)$.

For instance, note that when $S$ is finite then the bound $h(S)\le\#(S)$ is trivial. The original Helly number is $h(\R^d)=d+1$ and, interestingly, if $\F$ is any subfield of $\R$ (e.g., $\Q(\sqrt{2})$), then Radon's proof of Helly's theorem directly shows that the $S$-Helly number of $S=\F^d$ is still $d+1$. We know $h(\Z^{d-k}\times\R^k)=2^{d-k}(k+1)$. However, if $S$ is not necessarily a lattice but a general additive subgroup (e.g., $S=\{(\alpha \pi + \beta, \gamma) \in \R^2: \alpha,\beta,\gamma \in \Z\}$), then the $S$-Helly number is not covered by prior results. 
It is known that the $S$-Helly number may be infinite in some situations.
In Section \ref{sec:tools}, we state some prior results which are useful in computing $S$-Helly numbers.

\vskip .3cm
\noindent Now we are ready to state our contributions:

\subsection*{Bounds on Helly numbers for low dimensions and for algebraic subsets of \texorpdfstring{$\R^d$}{Rd}}
 
We present new Helly-type theorems for large classes of sets $S \subset \R^d$. For simplicity we will always assume that the linear span of $S$ is the whole space $\R^d$.

We begin in Section \ref{sec:lowdim} with results and open questions in dimensions one and two.
For the case of  dimension $d=1$ it is easy to see that for any set $S$, the Helly number $h(S)$ exists and is at most two.
In Section \ref{sec:lowdim} we present the following two theorems:

\begin{theorem} \label{thm:dim2}
If $S$ is a dense subset of $\R^2$, then $h(S)\leq 4$. This result is sharp.
\end{theorem}

\begin{theorem}\label{thm:lattice}
Let $L$ be a proper 
sublattice of $\Z^2$. If $S=\Z^2\setminus L$, then $h(S)\leq 6$. This result is sharp.
\end{theorem}

In Section \ref{sec:generaldim} we explore the situation where $S$ is algebraically structure set.
First any additive subgroup of $\R^d$ (not necessarily closed) and we propose an interesting conjecture for general subgroups. 
If $S$ has additional structure we can indeed obtain better bounds for $h(S)$.

\begin{theorem}\label{thm:modules}
Let $G$ be a dense subgroup of $\R$. If $S\subset \R^d$ is a $G$-module then $h(S)\le 2d$.
\end{theorem}

We then extend the classical Doignon's theorem for a lattice by proving that the set difference of a lattice and the union of several 
of its sublattices has bounded Helly number too. We prove the following theorem which has applications in convex discrete optimization (see \cite{chanceopt2015}):

\begin{theorem} \label{thm:difflat}
Let $L_1,\dots,L_k$ be (possibly translated) sublattices of $\Z^d$. Then the set $S =\Z^d \setminus (L_1\cup\dots\cup L_k)$ has Helly number $h(S)\le C_k2^d$ for some constant $C_k$ depending only on $k$.
\end{theorem}

Note that when no sublattices are removed this is exactly Doignon's theorem.
A quantitative version of Theorem \ref{thm:difflat}, with different bounds, was obtained in \cite{de2015quantitative}.
 

\subsection*{Finding Colorful versions}

Many Helly-type theorems admit a colorful version. In this setting, each convex set is 
assigned one of $N$ colors and every $N$ convex sets of distinct colors are required to intersect. 
The conclusion is that there is a color for which all convex sets intersect.
In Section \ref{sec:color} we analyze a popular method due to Lov\'asz that yields colorful versions of several Helly-type theorems.
As a result we find a few new colorful Helly-type theorems.

\section{Tools to compute Helly numbers}\label{sec:tools}

We will be concerned with proving bounds on $S$-Helly numbers of specific $S\subseteq\R^d$.
In this section we discuss some useful methods priorly  used to bound the $S$-Helly number of a subset $S$ of $\R^d$. 
They were studied first by Hoffman \cite{hoffman} in the setting of abstract convexity, and later on refined by 
Averkov and Weismantel in \cite{AW2012}, and Averkov in \cite{Ave2013}.

\begin{definition}
We define an \emph{$S$-vertex-polytope} as the convex hull of points $x_1,x_2,\dots,x_k\in S$ in convex position such that no other point of $S$ is in $\conv(x_1,\dots,x_k)$. 

Similarly, an \emph{$S$-face-polytope} is defined as the intersection of semi-spaces $H_1,H_2,\dots,H_k$ such that $\bigcap_i H_i$ has $k$ faces and contains exactly $k$ points of $S$, one contained in the relative interior of each face. 
\end{definition}
Figure \ref{fig:Spoly} shows an $S$-vertex-polytope with $6$ vertices and an $S$-face-polytope with six sides in $\R^2$.

\begin{figure}
\includegraphics{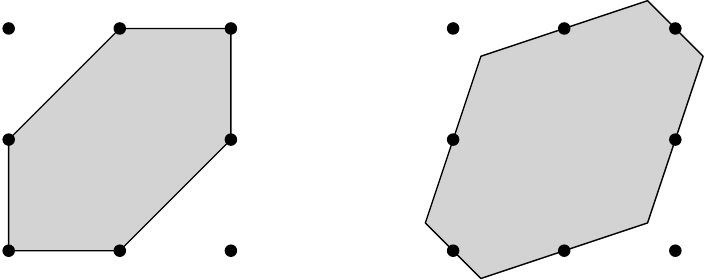}
\caption{An $S$-vertex-polytope and an $S$-face-polytope.}
\label{fig:Spoly}
\end{figure}

\begin{lemma}[Theorem 2.1 in \cite{Ave2013} and Proposition 3 in \cite{hoffman}]\label{lem:hollow}
Assume $S\subset\R^d$ is discrete, then $h(S)$ is equal to the following two numbers:
\begin{enumerate}
\item The supremum $f(S)$ of the number of faces of an $S$-face-polytope.
\item The supremum $g(S)$ of the number of vertices of an $S$-vertex-polytope.
\end{enumerate}
\end{lemma}

%
%

The case when $S$ is not discrete is a bit more delicate. Averkov \cite{Ave2013}  showed that in general $f(S) \leq h(S)$.
Recently, new results involving these two quantities and the $S$-Helly number $h(S)$ have been obtained by Conforti and Di Summa \cite{conforti-unp, Ave2013}. E.g., they proved that when $S$ is a closed subset and $f(S)$ is finite, then $h(S) \leq (d+1) f(S)$. They also showed an example where $f(S)=1$, yet the $S$-Helly number is infinite.   Fortunately, for the calculation of Helly numbers there is a very general result due to A.J. Hoffman characterizing $S$-Helly numbers.

\begin{lemma}[Proposition 2 in \cite{hoffman}]\label{lem:hof}
If $S\subset\R^d$, then $h(S)$ can be computed as the supremum of $h$ such that the following holds:
There exists a set $R=\{x_1,\dots,x_h\}\subset S$ such that $\bigcap_i\conv(R\setminus\{x_i\})$ does not intersect $S$.
\end{lemma}

Note that, in particular, the points of $R$ must be in strictly convex position. 
As a direct application of Lemma \ref{lem:hof} we have the following proposition.

\begin{prop}
If $S_1,S_2\subseteq\R^d$, then $h(S_1\cup S_2)\leq h(S_1)+h(S_2)$.
\end{prop}

The Helly number of a product is not always so well-behaved as in the mixed integer case, but Averkov and Weismantel  proved the following

\begin{prop}[{\cite[Thm. 1.1]{AW2012}}]
If $M$ is a closed subset of $\R^k$ then $h(\R^d\times M)\le (d+1)h(M)$.
\end{prop}

Averkov and Weismantel also proved that $h(\Z^d\times M)\ge 2^dh(M)$ for closed $M\subseteq\R^k$, but this is less useful for upper bounds.
Similarly, one can prove a general bound in the case of discrete sets.

 \begin{theorem}[Conforti, Di Summa \cite{conforti-unp}]
 	If $S_1, S_2 \subset \R^d$ are discrete sets, then $h(S_1 \times S_2) \ge h(S_1)h(S_2)$. 
 \end{theorem}

\section{\texorpdfstring{$S$}{S}-Helly numbers in low dimension}\label{sec:lowdim}

We start by looking at the one-dimensional case, here there is a very general result which immediatly follows from Lemma \ref{lem:hof}.

\begin{lemma}\label{lem:dim1}
If $S\subseteq\R$, then $h(S)=2$. 
\end{lemma}

That is, the $S$-Helly number of any subset of the real line is two.
Unfortunately in $\R^2$ we no longer have such a nice theorem. Consider the example below.

\begin{example}\label{ex:infinitehelly}
Let $S_n=\{p_1,\dots,p_n\}\subset\R^2$ be a set of $n$ points in strict convex position. Then any subset of $S_n$ can be expressed as $S_n \cap K$ where $K\subset\R^2$ is convex.
Lemma \ref{lem:hollow} implies $h(S_n)=n$. If for each $n$ we take a copy of $S_n$ such that their convex hulls do not intersect, then their union $S$ will have $h(S)=\infty$.
A simpler, but non-discrete, example with $h(S)=\infty$ is a circumference.
\end{example}

In spite of Example \ref{ex:infinitehelly}, we still have a general result in dimension two.

\begin{proof}[\bf Proof of Theorem \ref{thm:dim2}]
We begin by noting that the bound of four cannot be improved for arbitrary dense $S$. This can be done by taking $S=\R^2\setminus \{0\}$ and convex sets $\{x\ge 0\},\{x\le 0\},\{y\ge 0\},\{y\le 0\}$.

Now, assume that $h(S)\ge 5$. Lemma \ref{lem:hof} provides a set $R=\{x_1,\dots,x_5\}$ such that $\bigcap_i(R\setminus\{x_i\})$ does not intersect $S$. But since the points of $R$ are in strictly convex position, $\bigcap_i(R\setminus\{x_i\})$ has non-empty interior and must intersect $S$, a contradiction.%
%
%
\end{proof}

Theorem \ref{thm:dim2} does not hold in dimensions three and higher. In fact, we can construct a dense set $S$ in $\R^3$ with $h(S)=\infty$.

\begin{example}
In $\R^3$, consider a set of points $S_0$ on the plane $\{z=0\}$ with $h(S_0)=\infty$. Let $S=(\R^3\setminus\{z=0\})\cup S_0$. Then $S$ is clearly dense and $h(S)=\infty.$
\end{example}

One very important family of subsets of $\R^d$ are \emph{lattices}, i.e., discrete subgroups of $\R^d$. As we mention before, Doignon in \cite{Doi1973} made an interesting investigation of the Helly number of a lattice $S$ of rank $d$ and showed that $h(S)=2^d$ (his work was independently rediscovered by D. Bell and H. Scarf \cite{Bel1976,Sca1977}). Recently in \cite{ABDLL2014} this result was generalized to force not just an non-empty intersection with the lattice, but to control the number of lattice points in the intersection. Next we state a particular case of the results:

\begin{lemma}[See Theorem 1 in \cite{ABDLL2014}]\label{kpointslemma}
Let $k=0$, $1$ or $2$. Assume that $\F$ is a family of convex sets in $\R^2$ such that their intersection contains exactly $k$ points of $\Z^2$. Then there is a subfamily of $\F$ with at most $6$ elements such that their intersection contains exactly $k$ points of $\Z^2$.
\end{lemma}

The family consisting of the set difference between two lattices is a rather rich set of points in $\R^d$: they have interesting periodic patterns but contain complicated empty regions, and are closely related to tilings of space \cite{grunbaum+shephard}. See Figure \ref{fig} for some examples. 
Next we present a tight result in dimension two, which we generalize in Theorem \ref{thm:difflat} to arbitrary dimensions.

\begin{figure}
\includegraphics{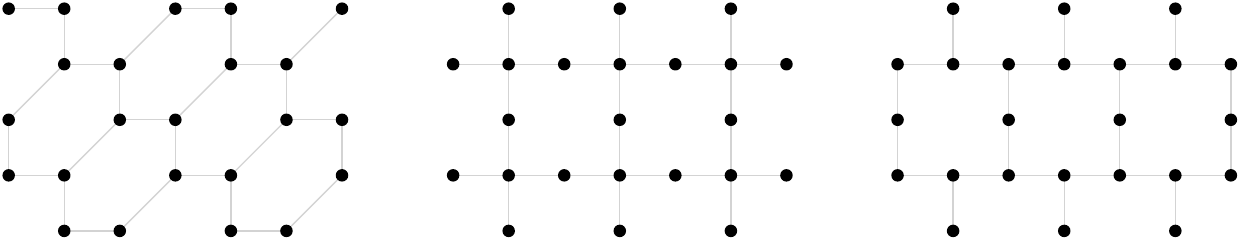}
\caption{Three examples of lattice differences.}
\label{fig}
\end{figure}

\begin{proof}[\bf Proof of Theorem \ref{thm:lattice}]
From Lemma \ref{lem:hollow} it is enough to bound the number of faces of an $S$-face-polygon. Assume $k>6$ and there is an $S$-face-polygon $K$ determined by the semiplanes $H_1,H_2,\dots,H_k$.

Note that if $\abs{\inte(K)\cap\Z^2}\leq 2$, then by Lemma \ref{kpointslemma} applied to the interiors of the $H_i$, we can find $H_{i_1},\dots,H_{i_6}$ among our original semiplanes such that $\inte(\bigcap_j H_{i_j})$ contains no additional points of $L$. But the relative interior of at least one side of $K$ is contained in $\inte(\bigcap_j H_{i_j})$, contradicting that there is a point of $S$ in each side of $K$. Therefore $\abs{K\cap L}\ge 3$. We look at two possible cases:

1.
Suppose there is a triplet of non collinear points in $K\cap L$. Then we may choose $p, q, r\in K\cap L$ such that $T=\text{conv}(p,q,r)$ contains no other point from $L$. Consequently, the area of $T$ is equal to $\frac{d(L)}2$, where $d(L)$ is the determinant of the lattice $L$. On the other hand, since $\inte(K)\cap S$ is empty, $\inte(K)\cap L=\inte(K)\cap \Z^2$ and therefore $T\cap\Z^2=\{p,q,r\}$. Hence the area of $T$ is also equal to $\frac{d(\Z^2)}2$. This contradicts the fact that $L$ is a proper sublattice of $\Z^2$.

2.
Suppose that the points of $\inte(K)\cap L$ are collinear.
After applying a suitable $\Z^2$-preserving linear transformation we may assume that $\inte(K)\cap L$ is contained in the $x$-axis.
Since $K$ intersects the $x$-axis in an interval with at least $3$ interior lattice points, its length is greater than $2$. Suppose $K$ contains a point with $y$-coordinate at least $2$. Then by convexity $K\cap\{y=1\}$ has width greater than $1$ and contains an interior lattice point $p$. Because the elements of $\inte(K)\cap L$ are collinear, $p$ cannot be in $L$, thus $p$ must be in $S$. However, this is impossible as $K$ is $S$-free. Therefore $K\subseteq\{\abs{y}<2\}$.
Each edge of $K$ contains a point from $S$ in its relative interior and every lattice point in $K$ has $y$-coordinate $-1$, $0$, or $1$. But a convex body can intersect each of the three lines $\{y=-1\}$, $\{y=0\}$, $\{y=1\}$ in at most two points. It follows that $K$ has at most six sides, contradicting the assumption that $K$ has at least seven sides.
\end{proof}

We conclude with a conjecture that stresses interesting connections to number theory.

\begin{conj}
Let $\mathbb P$ be the set of prime numbers. Then $h(\mathbb P^2)=\infty$.
\end{conj}

We have been able to show that $h(\mathbb P^2) \ge 14$ and $h((\Z \setminus \mathbb P)^2)$ is finite (e.g., through a simple modification of argument 
used in Theorem \ref{thm:difflat} and Lemma \ref{lem:hof}). This problem appears to be related to the Gilbreath-Proth conjecture \cite{gilbreath-proth} on the 
behavior of differences of consecutive primes.

\section{Helly Numbers for Subgroups of \texorpdfstring{$\R^d$}{R}  and differences of lattices}\label{sec:generaldim}

Now we move to spaces of arbitrary dimension and to sets $S$ with rich algebraic structure.

\subsection{Subgroups of \texorpdfstring{$\R^d$}{Rd}}
We look at the case when $S$ is an additive subgroup of $\R^d$. The most famous examples are lattices. We recall a lattice is defined as a discrete subgroup of $\R^d$; it is well known that lattices in $\R^d$ are generated by at most $d$ elements (see \cite{Bar2002}). This is precisely the context of Doignon's theorem, which says that $h(\Z^d)=2^d$.

If a group $S\subset\R^d$ is finitely generated by $m$ elements, then the natural epimorphism between $\Z^m$ and $S$ given by linear combinations yields a linear map from $\R^m$ onto $\R^d$. From here Doignon's theorem implies that $h(S)\le 2^m$. However, there is no dependency on $d$ here, and $m$ could be large compared to $d$. 

On the other hand, it is known in the topological groups literature (see e.g. \cite{Dik2013}) that every closed subgroup that linearly spans $\R^d$ is of the form $\phi(\R^k\times\Z^{d-k})$, where $\phi:\R^d\to\R^d$ is a linear bijection. In this case the mixed version of Helly's and Doignon's theorems by Averkov and Weismantel \cite{AW2012}, guaranties that $h(\Z^{d-k}\times\R^k)=2^{d-k}(k+1)$.

The closure $\bar{S}$ of a group $S$ is also a group; therefore we may assume that
$\bar S=\R^k\times\Z^{d-k}$. This allows us to express $S$ as a product of the form $\Z^{d-k} \times D$, where $D$ is a dense subgroup of $\R^k$. In this case we might expect that $h(S)\le (k+1)2^{d-k}$, but the following example shows that this is not the case.

\begin{example}
Let $k\leq d$ and $1, \alpha_1,\alpha_2,\dots \alpha_k$ be linearly independent numbers when considered as a vector space over $\Q$. Let $\{ e_1,e_2,\dots e_d\}$ be the canonical basis of $\R^d$ and consider the group
$$S=\langle e_1,\dots,e_{k-1},\alpha_1(e_{k}+e_1),\dots,\alpha_{k-1}(e_{k}+e_{k-1}),\alpha_ke_k, e_{k+1},\dots e_{d} \rangle.$$
Note that $S=D\times \Z^{d-k}$, where
$$D=\langle e_1,\dots,e_{k-1},\alpha_1(e_{k}+e_1),\dots,\alpha_{k-1}(e_{k}+e_{k-1}),\alpha_ke_k\rangle$$
is a dense set in $\R^d$, so the closure of $S$ is $\R^k\times\Z^{d-k}$. Observe that if we intersect $S$ with the space $\{x_k=0\}$ we obtain a $(d-1)$-dimensional lattice.
Since $$h(S\cap \{x_k=0\})=2^{d-1},$$ we can construct a family $\F$ with $2^{d-1}+2$ elements containing $\{x_k\ge0\}$ and $\{x_k\le0\}$ to show that $h(S)\ge 2^{d-1}+2$.
If $k\geq 3$, then $h(S)$ is larger than $(k+1)2^{d-k}$.
\end{example}

In spite of this, we state the following conjecture:

\begin{conj}
For any subgroup $G\subset \R^d$, the Helly number $h(G)$ is finite.
\end{conj}

Note that this conjecture is true for all finitely generated groups by Doignon's theorem.
One might be tempted to conjecture that $h(G)\le 2^d$. As we see in the next example, this is not always the case.

\begin{example}
Let $G_0\subset\R$ be the group generated by $1$, $\pi$ and $e$, so that $G=\Z^2\times G_0$ 
is a subgroup of $\R^3$. Consider the points $x_1=(0, 1, 3)$, $x_2=(2, 1, e)$ and $x_3=(2, 3, \pi)$ in $\R^3$. 
The important property these points have is that the $3$ midpoints they define are not in $G$.
Let $H\subset\R^3$ be the plane through $x_1,x_2,x_3$. Consider three additional points $x_4,x_5,x_6\in G$ 
above $(0,2),(1,0),(3,3)\in\Z^2$ and slightly above $H$ and another $3$ points $x_7,x_8,x_9$ above the same 
three points in $\Z^2$ but this time slightly below $H$. If $R=\{x_1,\dots,x_9\}$ then the hypothesis of 
Lemma \ref{lem:hof} are satisfied (see Figure \ref{fig:9example}), and therefore $h(G)\ge 9$.
\end{example}

\begin{figure}
\includegraphics{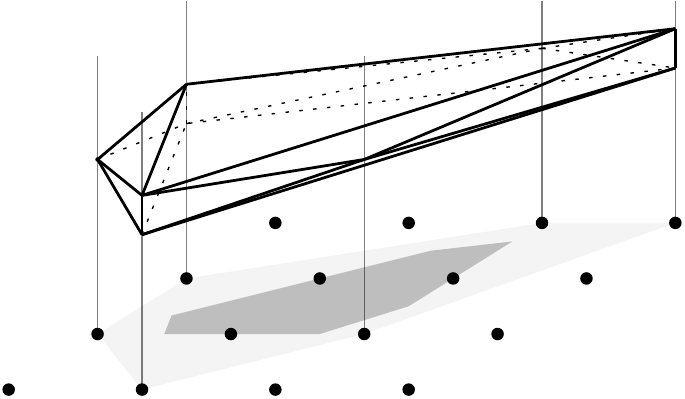}
\caption{The convex hull of $9$ points in $\Z^2\times G$, in dark gray is the projection of the intersection described in Lemma \ref{lem:hof}}.
\label{fig:9example}
\end{figure}

If the group $G$ has some additional structure, we are able to bound its Helly number.

\begin{proof}[\bf Proof of Theorem \ref{thm:modules}]
We have already proved in Lemma \ref{lem:dim1} the case $d=1$ and in Theorem \ref{thm:dim2} the case $d=2$. For the general case we use induction over $d$.

Let $d>1$; assume that the theorem is false and $h=h(S)>2d$.
This means that there is a family $\F$ of convex sets such that $\bigcap \F$ does not intersect $S$ but for every subfamily $\G$ with $h-1$ elements, $\bigcap\G$ intersects $S$.

Since $S$ is dense, $\inte(\bigcap \F)=\emptyset$; therefore $m=\dim(\bigcap \F)<d$. Consequently, by Steinitz's theorem \cite{Ste1913}, there is a subfamily $\cH\subset\F$ with $2(d-m)$ elements such that $\dim(\bigcap\cH)=\dim(\bigcap\F)$.

Since $2(d-m)<h$, there exists $s\in(\bigcap\mathcal H)\cap S$. 
Then, up to translation, $R=\text{span}(\bigcap\cH)\cap S$ is a group contained in an $m$-dimensional subspace of $\R^d$. The induction hypothesis on $R$ implies $h(R)\leq 2m$.
By the definition of $\F$, every $(h-1)-2(d-m)$ elements of $\F\setminus\cH$ intersect in $R$.
However, $\bigcap(\F\setminus\cH)$ does not intersect $R$, as then $\bigcap\mathcal F$ would intersect $R\subset S$. Thus $(h-1)-2(d-m)<h(R)$ and therefore $h\le h(R)+2(d-m)\le 2d$. This contradicts our choice of $h$.
\end{proof}

To show that this is best possible we have the following example.

\begin{example}
Let $a_1,a_2,\dots,a_d,b_1,b_2,\dots,b_d,c_1,c_2,\dots,c_d$ be linearly independent numbers when considered as a vector space over $\Q$ and satisfying $a_i<c_i<b_i$. Consider the $\Q$-module $S$ generated by the $2d$ vectors of the form
\begin{align*}
A_i&=(c_1,\dots,c_{i-1},a_i,c_{i+1},\dots,c_d),\\ B_i&=(c_1,\dots,c_{i-1},b_i,c_{i+1},\dots,c_d).
\end{align*}
It is easy to see that $C=(c_1,\dots,c_d)\not\in S$.

Let $K_i^-$ and $K_i^+$ be the two semispaces with boundary through $C$ orthogonal to $e_i$ so that $A_i\in K_i^-$ and $B_i\in K_i^+$. If $\F$ is the family consisting of these semispaces then $A_i\in\bigcap(\F\setminus\{K_i^+\})$ and $B_i\in\bigcap(\F\setminus\{K_i^-\})$ but $\bigcap\F=\{C\}$ does not intersect $S$; therefore $h(S)\ge 2d$.
\end{example}

\subsection{Difference of lattices in \texorpdfstring{$\R^d$}{Rd}}

Recall that the Ramsey number $$R_k=R(\underbrace{3,3,\dots,3}_k)$$
is the minimum natural number needed to guarantee the existence of a monochromatic triangle in any edge-coloring with $k$ colors of the complete graph with $R_k$ vertices. We now prove Theorem \ref{thm:difflat} 
with constant $C_k=R_k-1$.

\begin{proof}[\bf Proof of Theorem \ref{thm:difflat}]
Assume that $h(S)>(R_k-1)2^d$. Lemma \ref{lem:hollow} implies the existence of an $S$-vertex-polytope $K$ with $h(S)$ vertices. We say that two elements of $\Z^d$ have the same parity if their difference has only even entries, or equivalently, if their midpoint is in $\Z^d$. Since the vertex-set of $K$ has more than $(R_k-1)2^d$ elements, it contains a subset $V$ consisting of $R_k$ points with the same parity.

By definition of $K$, if $A,B\in V$ then their midpoint cannot be in $S$; it must be contained in some $L_i$. Consider $V$ as the vertex-set of the complete graph $G$ and assign to each edge $AB$ a color $i$ so that the midpoint of $A$ and $B$ is in $L_i$.

Since $G$ has $R_k$ vertices, by Ramsey's theorem it contains a monochromatic triangle. This means that there exist $A_1,A_2,A_3\in V\subset S$ such that the three midpoints $M_j=\frac12(A_j+A_{j+1})$ are in some $L_i$. But then $A_1=M_1-M_2+M_3\in L_i$ which contradicts the fact that $A_1\in S$.
\end{proof}

\section{Analysis of a popular method to color Helly-type theorems}\label{sec:color}

In this final section, we will see the colorful version of Helly's theorem. It can clearly be interpreted as a situation when the constraints are divided 
into colors, where the colors indicate constraint types or categories. What the theorem guarantees then is a situation in which at least one entire color class has a common solution. 
We discuss the essential features of a general method to obtain colorful versions of some Helly-type theorems in the sense of B\'ar\'any and Lov\'asz. This method has been used in a number of occasions \cite{AW2012,baranymatousekfracZ-helly} and can be applied to the usual Helly's theorem, Doignon's theorem, Theorem 1 in \cite{ABDLL2014} and our Theorems \ref{thm:lattice}  and \ref{thm:difflat} presented in this paper. This idea has been around for some time and it is based on Lov\'asz's proof of Helly's theorem, but as far as we know it has never been abstracted to its essential features.

Let $\Pp(K)$ be a property of a convex set $K\subset\R^d$. By a property we mean a function from the set of convex sets in $\R^d$ that takes values in $\{\text{true, false}\}$. As examples of properties $\Pp(K)$ we have:
\begin{enumerate}
\item[(i)] $K$ intersects a fixed set $S$,
\item[(ii)] $K$ contains at least $k$ integer lattice points,
\item[(iii)] $K$ is at least $k$-dimensional,
\item[(iv)] $K$ has volume at least $1$.
\end{enumerate}

The following generalizes the definition of Helly number we gave in the introduction to properties shared by sets.
\begin{definition}\label{def:colorhelly}
For a property $\Pp$ as above, the \emph{Helly number} $h=h(\Pp)\in\N$ is defined as the smallest number satisfying the following.
\begin{align} \label{helly:colorcond}
	&\forall i_1,\ldots,i_h \in [m] : \Pp(F_{i_1} \cap \cdots \cap F_{i_h}) && \Longrightarrow && \Pp(F_1 \cap \cdots \cap F_m)
\end{align}
for all $m \in\N$ and convex sets $F_1, \ldots, F_m$. If no such $h$ exists, then $h(\Pp):=\infty$.
\end{definition}

In particular, if $\Pp$ is the property (i) then $h(\Pp)=h(S)$. For property (ii), the quantitative Doignon theorem in \cite{ABDLL2014} gives bounds for $h(\Pp)$. A result from Gr\"unbaum \cite{Gru1962} bounds and determines $h(\Pp)$ for property (iii) with $0\le k\le d$. In the case of property (iv), it is easy to show that $h(\Pp)=\infty$, despite the existence of a quantitative Helly theorem \cite{Bar1982}.

\begin{definition}\label{def:color} We say that property $\Pp$ is:
\begin{enumerate}
\item[1.] \emph{Helly} if the corresponding Helly number $h(\Pp)$ is finite.
\item[2.] \emph{Monotone} if $K\subset K'$ and $\Pp(K)$ is true implies that $\Pp(K')$ is true too.
\item[3.] \emph{Orderable} if for any finite family $\F$ of convex sets there is a direction $v$ such that:
\begin{enumerate}
\item For every $K\in\F$ where $\Pp(K)$ is true, there is a containment-minimal $v$-semispace (i.e., a half-space of the form $\{x:v^T x\ge 0\}$) $H$ such that $\Pp(K\cap H)$ is true.
\item There is a unique minimal $K'\subset K\cap H$ with $\Pp(K')$ true.
\end{enumerate}
\item[4.] \emph{$N$-colorable} if for any given finite families $\F_1,\dots\F_N$ of closed convex sets in $\R^d$ (each a colored family), such that $\Pp(\bigcap\G)$ is true for every rainbow subfamily $\G$ (i.e. a family with $\G\cap\F_i=1$ for every $i$), then $\Pp(\bigcap\F_i)$ is true for some family $\F_i$.
\end{enumerate}
\end{definition}

It is not too difficult to see that property (i) is Helly, monotone and orderable when $S$ is discrete or the whole space. Property (ii) is also Helly, monotone and orderable, whereas property (iii) fails to be orderable and property (iv) is neither Helly nor orderable. Under this definition we can state the key result:

\begin{theorem}[Generic Colorful Helly]\label{teo:color}
A Helly, monotone, orderable property $\Pp$ is always $h(\Pp)$-colorable.
\end{theorem}

\begin{proof}[\bf Proof]
Let $\F$ be the family consisting of all the convex sets $\bigcap\G$ where $\G$ is a family of convex sets such that, for some $j$, $\G\cap\F_j=0$ and $\G\cap\F_i=1$ for every $i\neq j$.
Let $v$ be the direction provided by the hypothesis that the property is orderable.

By hypothesis and monotonicity, $\Pp(K)$ is true for every $K\in F$ and therefore there is a semispace $H_K$ as in condition (a) of orderability. Among these semispaces, set $H=H_K$ as the containment-maximal one and let $K'$ be as in condition (b) of orderability.
We may assume that $K=K_1\cap\dots\cap K_{N-1}$ with $K_i\in \F_i$ for $i=1,\dots,N-1$.

Take $K_N$ to be an arbitrary element of $\F_N$, we now show that $K'\subset K_N$.
Consider the family $\F'=\{K_1,\dots,K_{N-1},K_N,H_K\}$, it is not difficult to see that it satisfies \eqref{helly:colorcond} with $m=N$. Condition 1 of Definition \ref{def:color} implies that $\Pp(\bigcap\F')$ is true, but the uniqueness and minimality of $K'$ imply that $K'\subset K_N$. Since $K_N$ is arbitrary, we may conclude that $\Pp(\bigcap\F_N)$ is true.
\end{proof}

As a corollary we can obtain colored versions of three new Helly-type theorems: 

\begin{coro} Let  $\Pp$ be one of the following properties, all of which have an associated Helly-type theorem:
\begin{enumerate}

\item[(1.)] $K$ intersects a fixed set $S \subset \R^d$, where $h(S)<\infty$ and $S$ discrete.
\item[(2.)] $K$ contains at least $k$ integer lattice points (see \cite{ABDLL2014}),
\item[(3.)] $K$ contains at least $k$ points from a set difference of lattices (see \cite{de2015quantitative}),
\end{enumerate}
Then $\Pp$ is $h(\Pp)$-colorable.
\end{coro}

\section*{Acknowledgements} We are grateful to Andr\'es De Loera-Brust with his suggestions on computing the
Helly number of the prime grid. This research was supported by a UC MEXUS grant that 
helped established the collaboration of the UC Davis and UNAM teams. We are grateful for the support.
The first, second and third author travel was supported in part by the Institute for Mathematics  and its 
Applications and an NSA grant. The third and fourth authors  were also supported by CONACYT project 166306.

\bibliographystyle{amsplain}
\bibliography{justhelly}
\end{document}